\documentclass[11pt]{amsart}
\pdfoutput=1

\pdfpageattr{/Group <</S /Transparency /I true /CS /DeviceRGB>>}

\usepackage[T1]{fontenc}
\usepackage[english]{babel}
\usepackage{amsfonts,amsmath,amssymb,amsthm,mathrsfs}
\usepackage{tikz}
\usetikzlibrary{calc}
\usepackage{xspace}
\usepackage{marvosym}
\usepackage{microtype}
\usepackage{mathtools}
\usepackage[a4paper]{geometry}
\usepackage{hyperref}
\usepackage{float}

\microtypesetup{tracking,kerning,spacing}
\microtypecontext{spacing=nonfrench}

\usetikzlibrary{calc,3d,shapes,fit,arrows}
\usetikzlibrary{decorations.markings}

\tikzset{VertexStyle/.style = {
    circle, draw,
    text           = black,
    inner sep      = 2pt,
    outer sep      = 0pt,
    minimum size   = 20 pt}}

\tikzset{BlackDot/.style = {
    VertexStyle,
    fill = black}}
\tikzset{WhiteDot/.style = {
    VertexStyle,
    fill = white}}
\tikzset{RedDot/.style = {
    VertexStyle,
    fill = red}}
\tikzset{GreenDot/.style = {
    VertexStyle,
    fill = green}}
\tikzset{BlueDot/.style = {
    VertexStyle,
    fill = blue}}

\tikzset{ArcStyle/.style = {
    color          = black!70,
    thick,
    >=latex,
    ->
  }}

\tikzset{LabStyle/.style = {
    text = blue!80,
    fill=white
  }}

\tikzset{LoopStyle/.style = {text = blue!80}}

\tikzset{KoordStyle/.style = {->, >=latex, color=blue!40, thick}}
\tikzset{GitterStyle/.style = {very thin,color=blue!40}}

\tikzset{FillPoly/.style = {fill opacity = 0.3, fill=black}}
\tikzset{DrawPoly/.style = {very thick, draw=black}}
\tikzset{Poly/.style = {FillPoly, DrawPoly}}

\newcommand{\Gitter}[4]{
    \draw[GitterStyle,step=1.0] (#1,#3) grid (#2,#4);
}

\tikzstyle{lightgray} = [black:25]
\tikzstyle{blackdot} = [circle,draw,fill=black,scale=0.5]
\tikzstyle{whitedot} = [circle,draw,fill=white,scale=0.5]
\tikzstyle{edge} = [draw,-,black]
\tikzstyle{mediumedge} = [draw,thick,-,black,join=round]
\tikzstyle{thickedge} = [draw,ultra thick,-,black,join=round]

\newcommand\dotlabel[4]{
   \filldraw[#1] #2 circle (3pt);
   \draw #2 node [#3] {${#4}$};
}

\newcommand\polymake{{\tt polymake}\xspace}
\newcommand\atint{{\tt a-tint}\xspace}

\newcommand\cT{{\mathcal T}}

\newcommand\CC{{\mathbb C}}

\newcommand\KK{{\mathbb K}}

\newcommand\QQ{{\mathbb Q}}
\newcommand\RR{{\mathbb R}}
\newcommand\ZZ{{\mathbb Z}}
\newcommand\0{{\mathbf 0}}
\newcommand\1{{\mathbf 1}}

 \newcommand\norm[1]{|\hskip-.2ex|#1|\hskip-.2ex|}
\newcommand\SetOf[2]{\left\{#1\,\vphantom{#2}\right|\left.\vphantom{#1}\,#2\right\}}
\newcommand\smallSetOf[2]{\{#1 \colon #2\}}
\newcommand\Cayley{{\mathcal C}} 
\newcommand\Newton[1]{{\mathcal N}(#1)} 
\newcommand\extNewton[1]{\widetilde{\mathcal N}(#1)} 
\newcommand\dome[1]{{\mathcal D}\left(#1\right)} 
\newcommand\smalldowndome[1]{{\mathcal D}^{\downarrow}(#1)} 
\newcommand\downdome[1]{{\mathcal D}_{\downarrow}\!\left(#1\right)} 
\newcommand\updome[1]{{\mathcal D}^{\uparrow}\!\left(#1\right)} 
\newcommand\downSigma[1]{\Sigma_{\downarrow}\!\left(#1\right)}
\newcommand\smalldownSigma[1]{\Sigma_{\downarrow}(#1)}
\newcommand\upSigma[1]{\Sigma^{\uparrow}\!\left(#1\right)}

\newcommand\transpose[1]{{#1}^{\top}}
\newcommand\variety[1]{V(#1)} 
\newcommand\tropvariety[1]{\cT\left(#1\right)} 
\newcommand\smallmaxtropvariety[1]{\cT_{\max}(#1)} 
\newcommand\maxtropvariety[1]{\cT_{\max}\left(#1\right)} 

\newcommand\xx[1]{x_1^{\pm},x_2^{\pm},\dots,x_{#1}^{\pm}} 

\DeclareMathOperator\conv{conv}

\DeclareMathOperator\id{id}

\newcommand\TT{{\mathbb T}} 
\newcommand\maxTT{{\mathbb T}_{\max}} 
\newcommand\minTT{{\mathbb T}_{\min}} 
\newcommand\TP{{\mathbb{TP}}} 
\newcommand\minTP{{\mathbb{TP}}_{\min}} 
\DeclareMathOperator\trop{trop} 
\DeclareMathOperator\supp{supp} 
\newcommand\val{\operatorname{val}} 

\newcommand\puiseuxseries[2]{{#1}\{\hskip-.25em\{#2\}\hskip-.25em\}}

\DeclareMathOperator\tc{tc} 
\DeclareMathOperator\tconv{tconv} 
\newcommand\mintpos{\operatorname{tpos}_{\min}} 
\newcommand\mintconv{\operatorname{tconv}_{\min}} 
\newcommand\mintconvcirc{\operatorname{tconv}_{\min}^\circ} 
\newcommand\maxtconv{\operatorname{tconv}_{\max}} 
\newcommand\minodot{\odot_{\min}}
\newcommand\maxodot{\odot_{\max}}

\newcommand\maxoplus{\oplus_{\max}}

\newcommand\doi[1]{\href{http://dx.doi.org/#1}{\texttt{doi:#1}}}

\theoremstyle{plain}
    \newtheorem{theorem}{Theorem}
    \newtheorem{corollary}[theorem]{Corollary}
    \newtheorem{lemma}[theorem]{Lemma}
    \newtheorem{proposition}[theorem]{Proposition}
\theoremstyle{definition}
    \newtheorem{remark}[theorem]{Remark}
    \newtheorem{example}[theorem]{Example}

\graphicspath{ {img/} }

\title[The Cayley Trick for Tropical Hypersurfaces]{The Cayley Trick for Tropical Hypersurfaces\\ With a View Toward Ricardian Economics}

\author{Michael Joswig}

\address{
  Institut f{\"u}r Mathematik, MA 6-2,
  TU Berlin,
  Str.\ des 17. Juni 136, 10623~Berlin, Germany
}
\email{joswig@math.tu-berlin.de}

\thanks{This research is carried out in the framework of Matheon supported by Einstein Foundation Berlin. Further support by Deutsche
Forschungsgemeinschaft (Priority Program 1489: ``Experimental methods in algebra, geometry, and number theory'' and SFB/TR 109: ``Discretization in Geometry and Dynamics'')}

\dedicatory{Dedicated to Winfried Bruns on the occasion of his 70th birthday.}


\begin{document}

\begin{abstract}
  The purpose of this survey is to summarize known results about tropical hypersurfaces and the Cayley Trick from polyhedral geometry.
  This allows for a systematic study of arrangements of tropical hypersurfaces and, in particular, arrangements of tropical hyperplanes.
  A recent application to the Ricardian theory of trade from mathematical economics is explored.
\end{abstract}
\maketitle

\section{Introduction}
\noindent
The main motivation for this text are the applications of tropical geometry to economics, that came up recently.
In particular, Shiozawa gave an explanation of the Ricardian theory of international trade in terms of tropical combinatorics~\cite{Shiozawa:2015}.
The purpose of that theory is to study the relationship between wages and prices on the world market.
Our goal here is to put some of Shiozawa's results into the wider context of polyhedral and tropical geometry.

The Cayley Trick explains a special class of subdivisions of the Minkowski sum of finite point configurations in terms of a lifting to higher dimensions.
Those subdivisions are called \emph{mixed}.
Mixed subdivisions of Minkowski sums and mixed volumes play a key role in Bernstein's method for solving systems of polynomial equations.
Triangulations and more general polytopal subdivisions are the topic of the monograph \cite{Triangulations} by De Loera, Rambau and Santos.
In Section 1.3 of that book the relationship between systems of polynomials and mixed subdivisions is discussed.
Tropical geometry studies the images of algebraic varieties over fields with a discrete non-archimedean valuation under the valuation map; see Maclagan and Sturmfels \cite{Tropical+Book}.
Section~4.6 of that reference deals with a tropical version of Bernstein's Theorem, and this employs the Cayley Trick, too; see also Jensen's recent work on tropical homotopy continuation \cite{Jensen:1601.02818}.
A first version of the Cayley Trick was obtained by Sturmfels \cite{Sturmfels:1994}.
In its full generality it was proved by Huber, Rambau and Santos~\cite{HuberRambauSantos00}.

As its key contribution to tropical geometry the Cayley Trick explains how unions of tropical hypersurfaces work out.
It says that the union of two tropical hypersurfaces is dual to the mixed subdivision of the regular subdivisions which are dual to the two components.
This has been exploited by Develin and Sturmfels for the study of arrangements of tropical hyperplanes in the context of tropical convexity \cite{DevelinSturmfels04}.
More recently, those results have been extended by Fink and Rinc\'on \cite{FinkRincon:2015} and by Loho and the author \cite{JoswigLoho:1503.04707}.
It is this perspective which proves useful for applications to Ricardian economics.

Another recent application of tropical geometry to economics is Baldwin and Klemperer's study of ``product-mix auctions'' \cite{BaldwinKlemperer}.
There are $n$ indivisible goods, which are auctioneered in a one-round auction.
Each bidder gives bids (real numbers) for finitely many bundles of such goods (integer vectors of length $n$).
Aggregating all bundles of all bidders together with their bids leads to a mixed subdivision which is known as the ``demand complex''.
In contrast to the situation for the Ricardian economy, which is about tropical hyperplanes, i.e., tropical hypersurfaces of degree one, the tropical hypersurfaces that occur in product-mix auctions may have arbitrarily high degree.
While some of the results presented here do apply, product mix auctions themselves are beyond the scope of this survey.
In addition to the original~\cite{BaldwinKlemperer} the interested reader should consult Tran and Yu \cite{TranYu:1505.05737}.
In a similar vein Crowell and Tran studied applications of tropical geometry to mechanism design \cite{CrowellTran:1606.04880}.

I am indebted to Jules Depersin, Simon Hampe, Georg Loho, Yoshinori Shiozawa, and an anonymous referee for valuable discussions and comments.
The computations and the visualization related to the examples were obtained with \polymake \cite{DMV:polymake} and its extension \atint \cite{atint}.

\section{Regular and Mixed Subdivisions}
\noindent
We will start out with an explanation of the Cayley Trick.
Let $A$ be a finite set of points in $\RR^d$.
A \emph{(polyhedral) subdivision} of $A$ is a finite polytopal complex whose vertices lie in the set $A$ and that covers the convex hull $\conv A$.
For basic facts on the subject we refer to \cite{Triangulations}.
If $\lambda$ is any function that assigns a real number to each point in $A$, then the set
\begin{equation}\label{eq:extendedNewton}
  U(A,\lambda) \ := \ \conv\SetOf{(a,\lambda(a))\in\RR^d\times\RR}{a\in A} + \RR_{\geq 0} (\0,1)
\end{equation}
is an unbounded polyhedron in $\RR^{d+1}$; here ``+'' is the Minkowski addition, and $(\0,1)$ is the unit vector that indicates the ``upward'' direction.
Those faces of $U(A,\lambda)$ that are bounded admit an outward normal vector which points down, i.e., its scalar product with $(\0,1)$ is strictly negative.
Note that the outward normal vector of a facet is unique up to scaling.
On the other hand each lower-dimensional face has an entire cone of outward normal vectors, which is positively spanned by the outward normal vectors of the facets containing that face.
Projecting the bounded faces to $\RR^d$ by omitting the last coordinate defines a subdivision of $A$, which we denote as $\downSigma{A,\lambda}$.
A subdivision that arises in this way is called \emph{regular}.
The example where $\lambda(a)=\norm{a}^2$ is the Euclidean norm squared is the \emph{Delaunay} subdivision of $A$.

Now consider two finite subsets, $A$ and $B$, in $\RR^d$.
In the sequel we will be interested in special subdivisions of the Minkowski sum $A+B$.
Yet it will be important to address points in $A+B$ by their \emph{labels}.
That is, for distinct $a,a'\in A$ and distinct $b,b'\in B$ it may happen that $a+b=a'+b'$.
Nonetheless the label $(a,b)$ differs from the label $(a',b')$.
This means that the various labels of each point in $A+B$ keep track of the possibly many ways in which that point originates from $A$ and $B$.
For $A'\subset A$ and $B'\subset B$ the \emph{mixed cell} of $A+B$ with label set $A'\times B'$ is the polytope
\begin{equation}\label{eq:minkowski_cell}
  M(A',B') \ := \ \conv\SetOf{a+b}{a\in A',\, b\in B'} \enspace .
\end{equation}
Notice that the label set may also record points that are not vertices of $M(A',B')$.
A polyhedral subdivision of $A+B$ is \emph{mixed} if it is formed from mixed cells.

The \emph{Cayley embedding} of the point configurations $A$ and $B$ in $\RR^d$ is the point configuration
\begin{equation}\label{eq:cayley}
  \Cayley(A,B) \ := \ \SetOf{(a,-1)}{a\in A}\cup\SetOf{(b,1)}{b\in B}
\end{equation}
in $\RR^{d+1}$.
Any polytope of the form $\conv(\Cayley(A',B'))$ for subsets $A'\subset A$ and $B'\subset B$ is a \emph{Cayley cell}.
Intersecting the Cayley cell $\conv(\Cayley(A',B'))$ with the hyperplane $x_{d+1}=0$ yields the \emph{Minkowski cell} $M(A',B')$ with labeling $A'\times B'$.
Note that formally this does not quite agree with (\ref{eq:minkowski_cell}), not only because we identify $\RR^d$ with a linear hyperplane in $\RR^{d+1}$, but also because the intersection of the Cayley cell with that hyperplane  needs to be scaled by a factor of two to arrive at (\ref{eq:minkowski_cell}).
However, to avoid cumbersome notation, we ignore these details.
Let $\Sigma$ be any polyhedral subdivision of $\Cayley(A,B)$.
Then the set
\[
M(\Sigma) \ := \ \SetOf{M(A',B')}{\conv(\Cayley(A',B'))\in\Sigma}
\]
is a subdivision of the scaled Minkowski sum $\frac{1}{2}(A+B)$, and it is called the \emph{mixed subdivision} induced by $\Sigma$.
Again we will ignore the scaling factor, i.e., we will view $M(\Sigma)$ as a subdivision of $A+B$.

\begin{example}
  Let $A=\{0,1\}$ and $B=\{1,3\}$ be two pairs of points on the real line.
  The Cayley embedding $\Cayley(A,B)$ are the four vertices of the trapezoid shown in Figure~\ref{fig:cayley}.
  The two triangles
  \[
  \begin{aligned}
    \conv\Cayley(A,\{1\}) \ &= \ \conv\{(0,-1),(1,-1),(1,1)\} \quad \text{and} \\
    \conv\Cayley(\{1\},B) \ &= \ \conv\{(1,-1),(1,1),(3,1)\}
  \end{aligned}
  \]
  are Cayley cells, and they generate a subdivision of $\Cayley(A,B)$ which induces a mixed subdivision of the Minkowski sum $A+B=\{0+1,1+1,0+3,1+3\}$.
\end{example}

\begin{figure}[hb]\centering

\begin{tikzpicture}
\newcommand\height{2}

\coordinate (a) at (0,0);
\coordinate (b) at (1,0);

\coordinate (u) at (1,\height);
\coordinate (v) at (3,\height);

\foreach \first in {a,b}
  \foreach \second in {u,v}
     \coordinate (\first\second) at ($0.5*(\first)+0.5*(\second)$);


\draw[thickedge] (a) -- (b);
\node[below right] at (b) {$A\times\{-1\}$};

\draw[thickedge] (u) -- (v);
\node[below right] at (v) {$B\times\{1\}$};

\draw[mediumedge] (a) -- (u) -- (b) -- (v);

\foreach \point in {a,b,u,v}
  \filldraw[fill=black] (\point) circle (2pt);

\draw[thickedge] (au) -- (bv);
\node[below right] at (bv) {$\frac{1}{2}\cdot \bigl((A+B)\times\{0\}\bigr)$};

\foreach \point in {au,bu,bv}
  \filldraw[fill=white] (\point) circle (2pt);

\end{tikzpicture}
\caption{Cayley embedding of the two intervals $A=[0,1]$ and $B=[1,3]$}
\label{fig:cayley}
\end{figure}
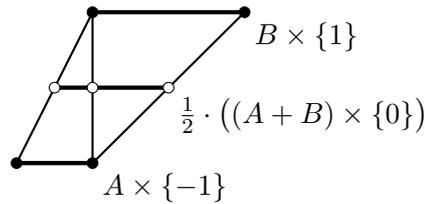

Notice that two Minkowski cells $M(A',B')$ and $M(A'',B'')$ intersect in the Minkowski cell $M(A'\cap A'',B'\cap B'')$ with labeling $(A'\cap A'')\times(B'\cap B'')$.
This consistency among the labels of the cells in a mixed subdivision allows to uniquely lift back any mixed subdivision to a subdivision of the Cayley embedding.
\begin{theorem}[{Cayley Trick \cite{Sturmfels:1994} \cite{HuberRambauSantos00}}]
  The map $M$ from the set of subdivisions of $\Cayley(A,B)$ to the set of mixed subdivisions of $A+B$ is a bijection that preserves refinement.
  Moreover, $M$ maps regular subdivisions to regular subdivisions.
\end{theorem}

Minkowski sums, mixed subdivisions and the Cayley Trick generalize to any finite number of point sets.
To this end assume that we have $n$ sets $A_1,A_2,\dots,A_n$ in $\RR^d$.
Then we pick an affine basis $u_1,u_2,\dots,u_n$ of $\RR^{n-1}$, i.e., the vertices of a full-dimensional simplex.
We define the \emph{Cayley embedding}
\[
\Cayley(A_1,A_2,\dots,A_n) \ := \ \SetOf{(a_1,u_1)}{a_1\in A_1}\cup\dots\cup\SetOf{(a_n,u_n)}{a_n\in A_n}
\]
in $\RR^d\times\RR^{n-1}$.
A particularly interesting case arises if we take $n$ copies of the same point set $A\in\RR^d$.
Then the Cayley embedding satisfies
\begin{equation}\label{eq:product}
  \Cayley(\underbrace{A,A,\dots,A}_{n \text{ times}}) \ \cong \ A \times \Delta_{n-1} \enspace ,
\end{equation}
where $\Delta_{n-1}=\conv(e_1,e_2,\dots,e_n)$ is the $(n{-}1)$-dimensional standard simplex in $\RR^n$.
Notice that we write ``$\cong$'' instead of ``='' since \eqref{eq:product} is only an affinely isomorphic image of what we defined in \eqref{eq:cayley}.
An in-depth explanation of the Cayley Trick can be found in~\cite[\S9.2]{Triangulations}.

\section{Tropical Hypersurfaces}
\noindent
Now we want to take a look into a few basic concepts from tropical geometry.
The Cayley Trick will prove useful to understanding unions of tropical hypersurfaces.

The \emph{tropical semiring} is the set $\TT=\RR\cup\{\infty\}$ equipped with $\min$ as the addition and $+$ as the multiplication.
The neutral element of the addition is $\infty$, and the multiplicative neutral element is $0$.
The tropical semiring behaves like a ring --- with the lack of additive inverses as the crucial exception.
If we want to stress the systematic role of these two arithmetic operations we write $\oplus$ instead of $\min$ and $\odot$ instead of $+$.
For further details on tropical geometry we refer to the monograph \cite{Tropical+Book} and the forthcoming book \cite{ETC}.

A \emph{tropical polynomial} is a formal linear combination of finitely many monomials (with integer exponents that may also be negative) in, say, $d$ variables with coefficients in $\TT$.
In this way a tropical polynomial $F$ gives rise to a function
\begin{equation}\label{eq:trop_polynomial}
  F(x) \ = \ \bigoplus_{m\in I} c_m \odot x^{\odot m} \ = \ \min_{m\in I}(c_m+m\cdot x) \enspace ,
\end{equation}
where $I$ is a finite subset of $\ZZ^d$ and the coefficients $c_m$ are elements of $\TT$.
By construction \eqref{eq:trop_polynomial} is a piecewise linear and concave function from $\RR^d$ to $\RR$.
The set $\supp(F)=\smallSetOf{m\in I}{c_m\neq\infty}$ is the \emph{support} of $F$.
Occasionally we will distinguish between formal tropical polynomials and tropical polynomial functions.
The set of formal tropical polynomials has a semiring structure where the addition and the multiplication is induced by $\oplus$ and~$\odot$.

We may read the support of a tropical polynomial as a point configuration that is equipped with a height function given by the coefficients, and this is what gives us a connection to the previous section.
The \emph{extended Newton polyhedron} of a tropical polynomial is a special case of \eqref{eq:extendedNewton}.
More precisely, if $F$ is defined as in \eqref{eq:trop_polynomial}, then we have
\[
\extNewton{F} \ := \ U(\supp(F),c) \enspace .
\]
Projecting the faces of $\extNewton{F}$ down yields the regular subdivision $\downSigma{F}:=\downSigma{\supp(F),c}$ of the support, and the convex hull is the \emph{Newton polytope} $\Newton{F}:=\conv(\supp(F))$.
It is worth noting that \emph{any} lifting function on \emph{any} finite set of lattice points can be read as a tropical polynomial.

One purpose of tropical geometry is to study classical algebraic varieties via their tropicalizations, which can be described in polyhedral terms.
Here we will restrict our attention to tropical hypersurfaces, which are the tropical analogs of the vanishing locus of a single classical polynomial.
The tropical polynomial $F$ \emph{vanishes} at $x\in\RR^d$ if the minimum in \eqref{eq:trop_polynomial} is attained at least twice, and the set
\[
\tropvariety{F} \ := \ \SetOf{x\in\RR^d}{F \text{ vanishes at } x}
\]
is the \emph{tropical hypersurface} defined by $F$.
It is immediate that $\tropvariety{F}$ is a polyhedral complex in $\RR^d$.
What may be less obvious is that this is a meaningful definition.
Yet the Fundamental Theorem of Tropical Geometry says that the tropical hypersurfaces are the images of classical varieties over a field with a non-Archimedean valuation (into the reals) under the valuation map; see Theorem~\ref{thm:fundamental} below.
However, we wish to postpone this discussion for a short moment, as we first want to introduce another polyhedron that we can associate with $F$; this is the \emph{dome}
\begin{equation}\label{eq:dome}
  \begin{split}
  \dome{F} \ :=& \ \SetOf{(p,s)\in\RR^{d+1}}{p\in\RR^d,\, s\in\RR,\, s\le F(p)} \\
  =& \ \bigcap_{m\in\supp(F)}\SetOf{(p,s)\in\RR^{d+1}}{s\leq c_m+m\cdot p} \enspace ,
\end{split}
\end{equation}
which is unbounded in the negative $e_{d+1}$-direction and of full dimension $d+1$.
Let $\updome{F}$ be the polyhedral complex that arises from $\partial\dome{F}$ by omitting the last coordinate, and we call this the \emph{normal complex} of the extended Newton polyhedron $\extNewton{F}$, or the \emph{normal complex} of $F$, for short.
This is a polyhedral subdivision of $\RR^d$ which is piecewise-linearly isomorphic to the boundary $\partial\dome{F}$ of the dome.
Now the tropical hypersurface $\tropvariety{F}$ is the codimension-$1$-skeleton of the normal complex $\updome{F}$, i.e., it corresponds to the codimension-$2$-skeleton of the polyhedron $\dome{F}$.
The latter is the set of faces whose dimension does not exceed $d-1$.
Summing up we have the following observation.
\begin{lemma}\label{lem:coarse}
  The facet defining inequalities of $\dome{F}$ correspond to certain points in the support of $F$.
  Furthermore, the facets of $\dome{F}$ are in bijection with the maximal cells of $\updome{F}$ as well as with the connected components of the complement of $\tropvariety{F}$ in $\RR^d$.
\end{lemma}
More precisely, using the notation of \eqref{eq:trop_polynomial} and \eqref{eq:dome}, the point $m\in\supp{F}$ yields a facet of $\dome{F}$ if there exists an $x\in\RR^d$ such that $F(x)=c_m\odot x^{\odot m}$ and $F(x)<c_{m'}\odot x^{\odot m'}$ for all $m'\neq m$.
In that case the inequality $s\leq c_m+m\cdot x$ is facet defining.
Now we want to relate the dome with the extended Newton polyhedron and the induced regular subdivision.
\begin{proposition}
  There is an inclusion reversing bijection between the face poset of $\dome{F}$ and the poset of bounded faces of $\extNewton{F}$.
  This entails that the tropical variety $\tropvariety{F}$ is dual to the regular subdivision $\downSigma{F}$ of $\supp{F}$.
\end{proposition}
Essentially this is a consequence of cone polarity.
Notice that the face poset of $\dome{F}$ is isomorphic with the poset of cells of the normal complex $\updome{F}$.

To explore the relationship of tropical with algebraic geometry here it suffices to consider one fixed field with a non-Archimedean valuation.
Its elements look as follows.
A \emph{formal Puiseux series} with complex coefficients is a power series of the form
\[
\gamma(t) \ = \ \sum_{k=m}^\infty a_k\cdot t^{k/N} \enspace ,
\]
where $m,N\in\ZZ$, $N>0$ and $a_k\in\CC$.
These formal power series with rational exponents can be added and multiplied in the usual way to yield an algebraically closed field of characteristic zero, which we denote as $\puiseuxseries{\CC}{t}$.
As a key feature there is a map $\val:\puiseuxseries{\CC}{t}\to\QQ$ that sends a Puiseux series to its lowest exponent.
This \emph{valuation map} satisfies
\[
\begin{aligned}
  \val(\gamma(t)+\delta(t)) \ &\geq \ \min\bigl(\val(\gamma(t)),\val(\delta(t))\bigr) \ = \ \val(\gamma(t)) \oplus \val(\delta(t)) \quad \text{and}\\
  \val(\gamma(t)\cdot\delta(t)) \ &= \ \val(\gamma(t))+\val(\delta(t)) \ = \ \val(\gamma(t)) \odot \val(\delta(t)) \enspace .
\end{aligned}
\]
We abbreviate $\puiseuxseries{\CC}{t}$ by $\KK$.
For any Laurent polynomial $f=\sum_{i\in I} \gamma_i(t)x^i$ in the ring $\KK[\xx{d}]$ its \emph{tropicalization} is the tropical polynomial
\[
\trop(f) \ := \ \bigoplus_{i\in I} \val(\gamma_i)\odot x^{\odot i} \enspace .
\]
The vanishing locus of $f$ is the hypersurface $\variety{f}:=\SetOf{x\in(\KK\setminus\{0\})^d}{f(x)=0}$ in the algebraic torus $(\KK\setminus\{0\})^d$.
In the sequel we consider a Laurent polynomial $f\in\KK[\xx{d}]$ and its tropicalization $F=\trop(f)$.
The following key result has been obtained by Kapranov; see~\cite{EinsiedlerKapranovLind06}.
\begin{theorem}[Fundamental Theorem of Tropical Geometry]\label{thm:fundamental}
  For every Laurent polynomial $f \in \KK[\xx{d}]$ we have
  \begin{equation}\label{eq:fundamental}
    \overline{\val(\variety{f})} \ = \ \tropvariety{\trop(f)} \enspace.
  \end{equation}
\end{theorem}
Here the valuation map $\val$ is applied element-wise and coordinate-wise to the points in the hypersurface $\variety{f}$, and here $\overline{\ \cdot \ }$ denotes the topological closure in $\RR^d$.
It should be noted that the Fundamental Theorem admits a generalization to arbitrary ideals in $\KK[\xx{d}]$; see \cite[\S3.2]{Tropical+Book}.
The hypersurface case corresponds to the principal ideals.
Now let us consider two Laurent polynomials $f,g\in\KK[\xx{d}]$ with tropicalizations $F=\trop(f)$ and $G=\trop(g)$.
\begin{lemma} \label{lem:union} 
  We have
  \[
  \tropvariety{F\odot G} \ = \ \tropvariety{\trop(f\cdot g)} \ = \ \tropvariety{F}\cup\tropvariety{G} \enspace .
  \]
\end{lemma}
\begin{proof}
  A direct computation shows that $\trop(f\cdot g)$ equals $F\odot G$.
  As $\variety{f\cdot g}=\variety{f}\cup\variety{g}$ holds classically, the claim follows from Theorem~\ref{thm:fundamental}.
\end{proof}
The next result says that tropicalization commutes with forming unions of (tropical) hypersurfaces.
\begin{proposition} \label{prop:union_diagram}
  The diagram
  \begin{equation}\label{eq:puiseux:union_diagram}
    \begin{tikzpicture}[%
  node distance=2.75cm, auto,
  vertical/.style={node distance=2cm}	
]

\node (Vf) {$\variety{f}$};
\node (Vfg) [right of=Vf] {$\variety{f\cdot g}$};
\node (Vg) [right of=Vfg]{$\variety{g}$};

\draw[->] (Vf) -- (Vfg);
\draw[->] (Vg) -- (Vfg);

\node (Tf) [vertical, below of=Vf]{$\tropvariety{F}$};
\node (Tfg) [vertical, below of=Vfg] {$\tropvariety{F\odot G}$};
\node (Tg) [vertical, below of=Vg]{$\tropvariety{G}$};

\draw[->] (Tf) -- (Tfg);
\draw[->] (Tg) -- (Tfg);

\draw[->] (Vf) to node {$\val$} (Tf);
\draw[->] (Vfg) to node {$\val$} (Tfg);
\draw[->] (Vg) to node {$\val$} (Tg);

\node (Df) [vertical, below of=Tf]{$\partial\dome{F}$};
\node (Dfg) [vertical, below of=Tfg] {$\partial\dome{F\odot G}$};
\node (Dg) [vertical, below of=Tg]{$\partial\dome{G}$};

\draw[->] (Df) to node {$\odot G$} (Dfg);
\draw[<-] (Dfg) to node {$\odot F$} (Dg);

\draw[->] (Tf) to node {$\id\times F$} (Df);
\draw[->] (Tfg) to node {$\id\times(F\odot G)$} (Dfg);
\draw[->] (Tg) to node {$\id\times G$} (Dg);

\end{tikzpicture}
  \end{equation}
  commutes.
  The map $\odot G$ sends a point $(w,s)\in\RR^{d+1}$ to $(w,s+G(w))$, and $\odot F$ is similarly defined.
  The unmarked horizontal arrows are embeddings of subsets.
\end{proposition}
\begin{proof}
  The upper two squares in the diagram commute due to the Fundamental Theorem.
  We focus on the lower left square; the lower right one is similar.
  Let $w\in \tropvariety{F}$.
  The latter is contained in $\tropvariety{F\odot G}=\tropvariety{F}\cup\tropvariety{G}$ by Lemma~\ref{lem:union}.
  Evaluating $F$ at $w$ yields the point $(w,F(w))$ in the codimension-$2$-skeleton of the dome $\dome{F}\subset\RR^{d+1}$, which is part of the boundary.
  Any point in the boundary of $\dome{F}$ has the form $(v,F(v))$ for some $v\in\RR^d$.
  We can check that
  \begin{equation}\label{eq:union:polyhedra}
    \odot G(v,F(v)) \ = \ (v,F(v)+G(v)) \ = \ (v,F\odot G(v)) \enspace ,
  \end{equation}
  and thus $\odot G$, indeed, maps arbitrary points in the boundary of $\dome{F}$ to boundary points of $\dome{F\odot G}$.
  Setting $v=w$ in (\ref{eq:union:polyhedra}) now finishes the proof.
\end{proof}
\begin{remark}\label{rem:refinement}
  From (\ref{eq:union:polyhedra}) it also follows that the normal complex $\updome{F\odot G}$ is the common refinement of the normal complexes $\updome{F}$ and $\updome{G}$.
\end{remark}

The vertices of the regular subdivision $\downSigma{F\odot G}$ are sums of one point in $\supp(F)$ with one point in $\supp(G)$, i.e., they correspond to products of a monomial in $f$ with a monomial in $g$.
Altogether the Cayley embedding of the monomials of the factors, seen as configurations of lattice points, project to the monomials in the product.
Now, via the Cayley Trick, any regular subdivision of the Cayley embedding induces a coherent subdivision of the Minkowski sum.
\begin{corollary}\label{cor:downmixed}
  The regular subdivision $\downSigma{F\odot G}$ is a coherent mixed subdivision of $\supp(F)+\supp(G)$.
\end{corollary}

Classically, varieties defined by homogeneous polynomials are studied in the projective space.
Here the situation is similar.
A tropical polynomial $F$ is \emph{homogeneous} of degree $\delta$ if its support is contained in the affine hyperplane $\sum x_i=\delta$.
For such $F$ the tropical hypersurface $\tropvariety{F}$ can be seen as a subset of the \emph{tropical projective torus} $\RR^d/\RR\1$.
That set is homeomorphic with $\RR^{d-1}$, and it has a natural compactification, the \emph{tropical projective space}
\[
\TP^{d-1} \ := \ \left(\bigl(\RR\cup\{\infty\}\bigr)^d\setminus\{\infty\1\}\right)/\RR\1 \enspace .
\]

Below we will also look into the set $\maxTT=\RR\cup\{-\infty\}$ with $\max$ as the addition instead of $\min$.
These two versions of the tropical semiring are isomorphic via $-\min(x,y)=\max(-x,-y)$.
Hence the results of this section, suitably adjusted, also hold for max-tropical polynomials.
To avoid confusion we will use ``min'' or ``max'' as subscripts wherever necessary.
As far as the regular subdivisions are concerned, for min we look at the lifted points from ``below'', while for max we look from ``above''.
To mark this difference we write $\upSigma{H}$ and $\downdome{H}$, if $H$ is a max-tropical polynomial, and we have
\begin{equation}\label{eq:updown}
  \upSigma{H} \ = \ \downSigma{-H} \quad \text{and} \quad \downdome{H} \ = \ -\updome{-H} \enspace ,
\end{equation}
where $-H$ is the min-tropical polynomial that arises from $H$ be replacing each coefficient by its negative, and the minus in front of the polyhedral complex on the right refers to reflection at the origin of $\RR^d$.

\section{Arrangements of Tropical Hyperplanes}
\label{sec:arrangements}
\noindent
The simplest kind of algebraic hypersurfaces are the hyperplanes, i.e., the linear ones.
Arrangements of hyperplanes is a classical topic with a rich connection to algebraic geometry, group theory, topology and combinatorics.
A standard reference is the monograph \cite{OrlikTerao:1992} by Orlik and Terao.
The tropicalization of hyperplane arrangements was pioneered by Develin and Sturmfels~\cite{DevelinSturmfels04}.
The Cayley Trick will sneak into the discussion through Corollary~\ref{cor:downmixed}.

Let $V=(v_{ik})\in\minTT^{d\times n}$ be a matrix whose coefficients are real numbers or $\infty$.
Throughout we will assume that each column contains at least one finite entry.
Writing $v^{(k)}$ for the $k$-column this means that $v^{(k)}+\RR\1$ is a point in the tropical projective space $\minTP^{d-1}$.

The \emph{negative} column vector $-v^{(k)}$ is an element of $(\RR\cup\{-\infty\})^d$, and it defines the homogeneous \emph{max}-tropical linear form
\[
\begin{split}
  (-v_{1k})\odot x_1 \ \maxoplus \ (-v_{2k})\odot x_2 \ \maxoplus \ \cdots \ \maxoplus \ (-v_{dk})\odot x_d \\
  = \ \max\{ x_1-v_{1k}, x_2-v_{2k}, \dots, x_d-v_{dk}\} \enspace ,
\end{split}
\]
which we will identify with $-v^{(k)}$.
Since we assumed that $v^{(k)}$ has at least one finite coefficient that tropical linear form is not trivial.
The tropical variety $\smallmaxtropvariety{-v^{(k)}}$ is a max-tropical hyperplane and, by Lemma~\ref{lem:union}, the tropical variety associated with the product of linear forms
\[
\begin{split}
  \maxtropvariety{-V} \ :=& \ \maxtropvariety{(-v^{(1)}) \odot (-v^{(2)}) \odot \dots \odot (-v^{(n)})} \\
  =& \ \smallmaxtropvariety{-v^{(1)}} \cup \smallmaxtropvariety{-v^{(2)}} \cup \dots \cup \smallmaxtropvariety{-v^{(n)}}
\end{split}
\]
is a union of tropical hyperplanes.

The support $\supp(-v^{(k)})$ is a subset of the vertices of the ordinary standard simplex $\Delta_{d-1}=\conv(e_1,e_2,\dots,e_d)$ in $\RR^d$.
So the Newton polytope of the tropical linear form $-v^{(k)}$ is a face of $\Delta_{d-1}$
If that Newton polytope is the entire simplex of dimension $d-1$ then the normal complex $\updome{v^{(k)}}$ has a unique vertex that is contained in $d$ maximal cells which are dual to the vertices of $\Delta_{d-1}$.
Notice that, following standard practice in polyhedral geometry \cite{Triangulations}, even in the max-tropical setting, we usually prefer to look at regular subdivisions from ``below'', and \eqref{eq:updown} takes care of the translation.
Thus we study $\updome{v^{(k)}}$ rather than its image $\downdome{-v^{(k)}}$ under reflection.
The following result is a consequence of the Cayley Trick in the guise of Corollary~\ref{cor:downmixed}.

\begin{proposition}\label{prop:downsigma}
  The regular subdivision $\downSigma{V}$, which is dual to the $\max$-tropical hypersurface $\maxtropvariety{-V}$, coincides with the mixed subdivision
  \[
  M\bigl(\smalldownSigma{v^{(1)}},\smalldownSigma{v^{(2)}},\dots,\smalldownSigma{v^{(n)}}\bigr) \enspace .
  \]
\end{proposition}

\begin{figure}[tb]\centering
  \includegraphics[height=.4\textheight]{trop-polytope-3.mps}

  \caption{Four max-tropical hyperplanes in $\RR^3/\RR\1$}
  \label{fig:arrangement}
\end{figure}


\begin{example}\label{exmp:arrangement}
  For
  \[
  V \ = \ \begin{pmatrix} 0 & 0 & 0 & 0\\ 1 & 4 & 3 & 0\\ 0 & 1 & 3 & 2\end{pmatrix}
  \]
  the tropical hypersurface $\maxtropvariety{-V}$ is a union of four tropical lines in the tropical projective $2$-torus.
  The arrangement and the normal complex $\downdome{-V}$ is displayed in Figure~\ref{fig:arrangement}.
  In that drawing each point $(x_1,x_2,x_3)+\RR\1\in\RR^3/\RR\1$ occurs as $(x_2-x_1,x_3-x_1)$.

  The Newton polytope of each max-tropical linear form $-v^{(k)}$ is the standard triangle $\Delta_2$.
  The fourfold Minkowski sum is the dilated triangle $4{\cdot}\Delta_2$.
  The mixed subdivision $\upSigma{-V}$ is shown in Figure~\ref{fig:mixed}.
  That picture also shows the tropical line arrangement from Figure~\ref{fig:arrangement} embedded into the dual graph of the subdivision.

  According to Lemma~\ref{lem:coarse} each connected component of the complement of  $\maxtropvariety{-V}$ is marked with the corresponding element from the support set of the max-tropical polynomial
  \begin{equation}\label{eq:arrangement}
    \begin{split}
      (-v^{(1)}) \odot (-&v^{(2)}) \odot  (-v^{(3)}) \odot (-v^{(4)})\\
      = \ \max( & 4x_1 ,\, 3x_1+x_2 ,\, 3x_1+x_3 ,\, 2x_1+2x_2-1 ,\, 2x_1+x_2+x_3 ,\, 2x_1+2x_3-1 ,\,\\ & x_1+3x_2-4 ,\, x_1+2x_2+x_3-2 ,\, x_1+x_2+2x_3-1 ,\, x_1+3x_3-3 ,\,\\ & 4x_2-8 ,\, 3x_2+x_3-5 ,\, 2x_2+2x_3-4 ,\, x_2+3x_3-4 ,\, 4x_3-6 ) \enspace .
    \end{split}
  \end{equation}
  For instance evaluating the max-tropical polynomial \eqref{eq:arrangement} at the point $(0,2,0)$ yields the maximum $3$, which is attained at the tropical monomial $2x_1+2x_3-1$ with exponent vector $(2,2,0)$.
  In the language of \cite{DochtermannJoswigSanyal:2012} that vector is the ``coarse type'' of the corresponding (maximal) cell of $\upSigma{-V}$; see Remark~\ref{rem:separate} below.
  In the dual mixed subdivision we actually see $(2,2,0)$ as the coordinates of the lattice point dual to the maximal cell of $\downdome{-V}$ that contains the point $(0,2,0)$ in its interior.
\end{example}

\begin{figure}[tb]\centering
  \includegraphics[height=.4\textheight]{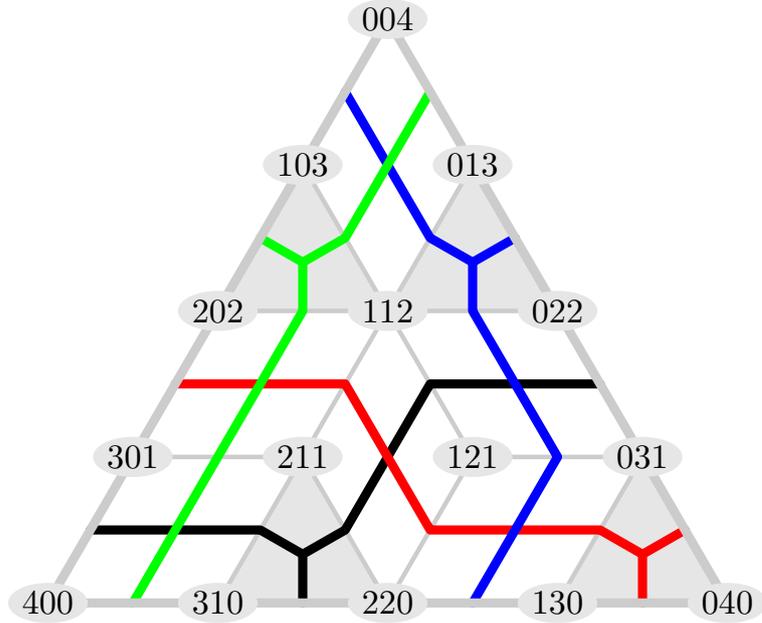}
  \caption{Mixed subdivision of $4{\cdot}\Delta_2$}
  \label{fig:mixed}
\end{figure}

We now turn to investigating a tropical version of convexity.
The set
\[
\mintpos(V) \ := \ \SetOf{\bigoplus_{k=1}^n \lambda_k \odot v^{(k)}}{\lambda_k\in\TT} \ = \ \SetOf{V\minodot\lambda}{\lambda\in\TT^n}
\]
is the $\min$-\emph{tropical cone} spanned by (the columns of) $V$.
It satisfies $\mintpos(V)=\mintpos(V)+\RR\1$, which is why it can be studied as a subset of the tropical projective space $\minTP^{d-1}$.
The image $\mintconv(V)$ of $\mintpos(V)$ under the canonical projection $\minTT^d\to\minTP^{d-1}$ is called a \emph{tropical polytope}.
In the sequel we will concentrate on the intersection
\[
\mintconvcirc(V) \ := \ \mintconv(V)\cap(\RR^d/\RR\1) \enspace ,
\]
which comprises the points with finite coordinates in the tropical polytope $\mintconv(V)$.

\begin{theorem}\label{thm:tconv}
  The set $\mintconvcirc(V)$ is a union of cells of the polyhedral complex
  \[
  \downdome{(-v^{(1)}) \odot (-v^{(2)}) \odot \dots \odot (-v^{(n)})}
  \]
  in $\RR^d/\RR\1$.
  If all coefficients of $V$ are finite then $\mintconvcirc(V)=\mintconv(V)$ is the union of those cells that are bounded.
\end{theorem}
Theorem~\ref{thm:tconv} was proved by Develin and Sturmfels \cite{DevelinSturmfels04} for finite coefficients.
Extensions to the general case have been obtained by Fink and Rinc\'on~\cite{FinkRincon:2015} and by Loho and the author~\cite{JoswigLoho:1503.04707}.
As a consequence of Remark~\ref{rem:refinement} the normal complex $\downdome{(-v^{(1)}) \odot (-v^{(2)}) \odot \dots \odot (-v^{(n)})}$ in $\RR^d/\RR\1$ is the common refinement of the $n$ normal complexes $\smalldowndome{-v^{(1)}}$, $\smalldowndome{-v^{(2)}}$, $\ldots$, $\smalldowndome{-v^{(n)}}$; see Figure~\ref{fig:arrangement}.

\begin{figure}[tb]\centering
  \resizebox{!}{.4\textheight}{\begingroup

\newcommand\xmin{-1}
\newcommand\xmax{5}
\newcommand\ymin{-1}
\newcommand\ymax{4}

\begin{tikzpicture}[ x={(1.6cm,0cm)}, y={(0cm,1.6cm)} ]
  \Gitter{\xmin+0.1}{\xmax-0.1}{\ymin+0.1}{\ymax-0.1}
  \foreach \x in {0,...,4} {
    \draw[GitterStyle] (\x,\ymin) node [below] {$\x$};
  }
  \foreach \y in {0,...,3} {
    \draw[GitterStyle] (\xmin,\y) node [left] {$\y$};
  }

  \coordinate (p1) at (0,0);
  \coordinate (p2) at (1,0); 
  \coordinate (p3) at (2,1); 
  \coordinate (p4) at (3,1);
  \coordinate (p5) at (4,1); 
  \coordinate (p6) at (3,2);
  \coordinate (p7) at (3,3); 
  \coordinate (p8) at (1,3);
  \coordinate (p9) at (0,2); 
  \coordinate (p10) at (0,1);
  
  \filldraw[Poly] (p1) -- (p2) -- (p3) -- (p4) -- (p7) -- (p8) -- (p9) -- cycle;
  \draw[DrawPoly] (p10) -- (p3) -- (p6);
  \draw[DrawPoly] (p4) -- (p5);

  \dotlabel{BlackDot}{(p2)}{below right}{v^{(1)}}
  \dotlabel{RedDot}{(p5)}{above right}{v^{(2)}}
  \dotlabel{BlueDot}{(p7)}{above right}{v^{(3)}}
  \dotlabel{GreenDot}{(p9)}{above left}{v^{(4)}}
\end{tikzpicture}
\endgroup

}
  \caption{Tropical polytope generated by four points in $\RR^3/\RR\1$}
  \label{fig:polygon}
\end{figure}
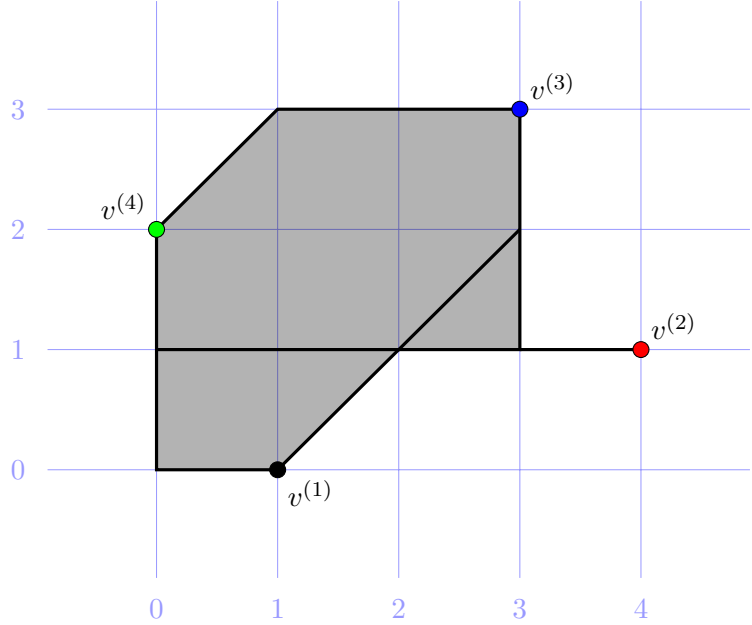

\begin{example}
  For the matrix $V$ from Example~\ref{exmp:arrangement} the tropical polytope $\mintconv(V)$ agrees with $\mintconvcirc(V)$, and it is shown in Figure~\ref{fig:polygon}.
  There are four bounded cells in the subdivision $\upSigma{-V}$, shown in Figure~\ref{fig:arrangement}, which are maximal with respect to inclusion, three of dimension two and one of dimension one.
  These form the tropical convex hull.
\end{example}

We want to come back to the Cayley Trick.
If all entries of the matrix $V\in\TT^{d\times n}$ are finite then the Newton polytope of the linear forms corresponding to each of the $n$ columns is the simplex $\Delta_{d-1}$.
In this case the Cayley embedding $\Cayley(\Delta_{d-1},\Delta_{d-1},\dots,\Delta_{d-1})$ is isomorphic with the product $\Delta_{d-1}\times\Delta_{n-1}$ of simplices.
\begin{corollary}\label{cor:product}
  If all entries of $V$ are finite then the regular mixed subdivision of $n{\cdot}\Delta_{d-1}$ from Proposition~\ref{prop:downsigma} is piecewise linearly isomorphic with a slice of the regular subdivision of $\Delta_{d-1}\times\Delta_{n-1}$ where the vertex $(e_i,e_k)$ is lifted to $v_{ik}$.
\end{corollary}

Clearly, when we talk about subdivisions of products of simplices it makes sense to think about exchanging the factors.
A direct computation shows that this corresponds to changing from $V\in\RR^{d\times n}$ to the transpose $\transpose{V}\in\RR^{n\times d}$.
Figure~\ref{fig:dual} shows the max-tropical hyperplane arrangement in $\RR^4/\RR\1$ and the min-tropical convex hull arising from the $4{\times}3$-matrix $\transpose{V}$ for $V$ as in Example~\ref{exmp:arrangement}.
The mixed subdivision $\upSigma{-\transpose{V}}$ is displayed in Figure~\ref{fig:mixed-3d}.
\begin{figure}[tb]\centering
  \includegraphics[height=.4\textheight]{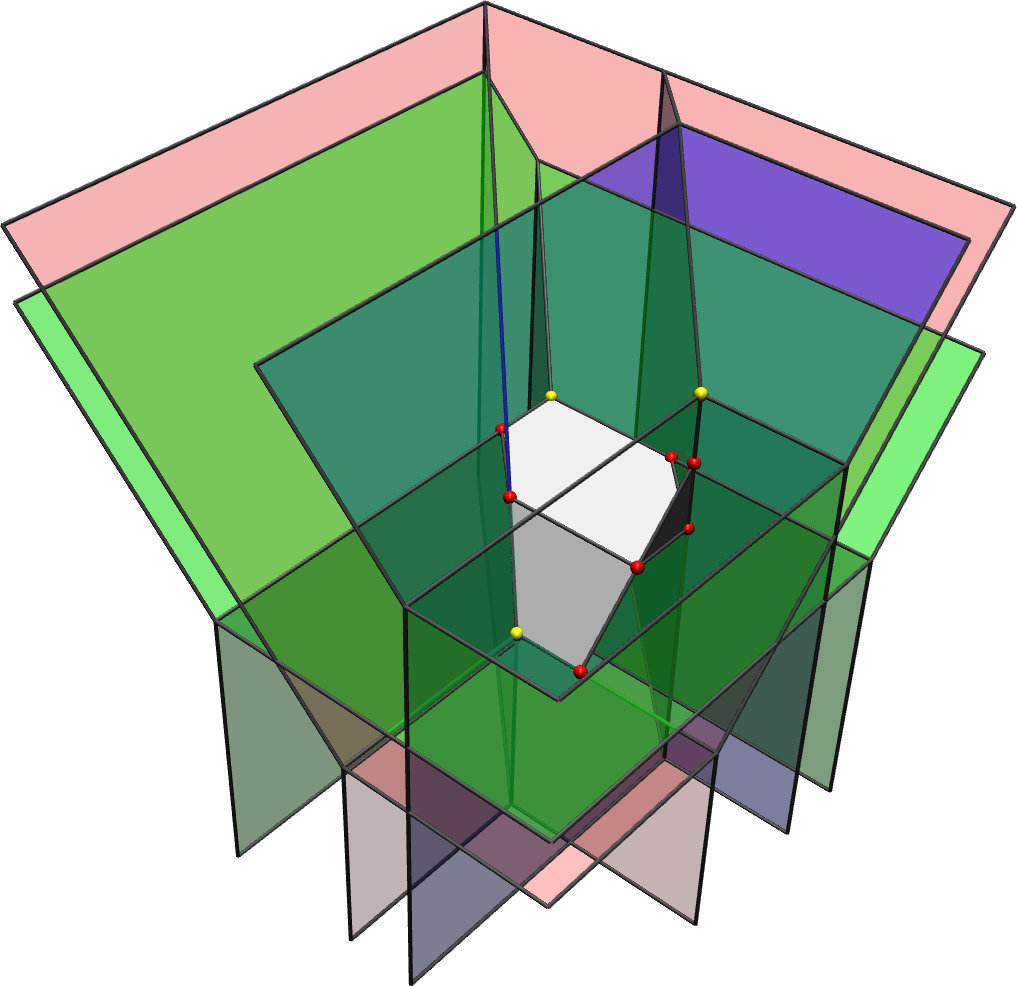}
  \caption{Three max-tropical hyperplanes in $\RR^4/\RR\1$. Compare with Figure~\ref{fig:polygon}}
  \label{fig:dual}
\end{figure}

If the matrix $V$ contains at least one coefficient $\infty$ then Corollary~\ref{cor:product} holds for the proper subpolytope
\[
\conv\SetOf{(e_i,e_k)\in\RR^d\times\RR^n}{v_{ik}\in\RR}
\]
of $\Delta_{d-1}\times\Delta_{n-1}$.
These subpolytopes and their subdivisions are studied in \cite{FinkRincon:2015} and~\cite{JoswigLoho:1503.04707}.
The \emph{tropical covector} of a point $z\in\RR^d$ with respect to the matrix $V$ is defined as
\[
\tc(z) \ := \ \SetOf{(i,k)\in[d]\times[n]\vphantom{\strut}}{v_{ik}-z_i=\min(v_{1k}-z_1,v_{2k}-z_2,\dots,v_{dk}-z_d)} \enspace .
\]
That is to say, the pair $(i,k)$ lies in $\tc(z)$ if and only if the minimum of the coordinates of the difference $v^{(k)}-z$ is attained at $i$.
This encodes the relative position of $z$ with respect to the columns of $V$.
It is immediate that $\tc(z)$ is constant on the set $z+\RR\1$.
Thus it is well-defined for points in the tropical projective torus.
In the language of \cite{DevelinSturmfels04} the tropical covectors occur as ``types'', and these are called ``fine types'' in \cite{DochtermannJoswigSanyal:2012}.
The term ``covector'' was first used in \cite{FinkRincon:2015} and subsequently in \cite{JoswigLoho:1503.04707}.

\begin{table}[H]
  \caption{Max-tropical polynomial which arises as the product of four tropical linear forms with separate variables $y_{ij}$ for $(i,j)\in[3]\times[4]$.
    Evaluating at $y_{1j}=0$, $y_{2j}=1$, $y_{3j}=3$, for $j=1,2,3,4$, yields the value $6$, which is the maximum taken over $81=3^4$ terms.
    The four terms for which that maximum is attained are marked.}
  \label{tab:poly}
  \[
  \begin{split}
    W(&y_{11},y_{12},\dots,y_{34}) \ = \ \max ( \\
    &y_{11}+y_{12}+y_{13}+y_{14} , \  y_{11}+y_{12}+y_{13}+y_{24} , \  -2+y_{11}+y_{12}+y_{13}+y_{34} , \  \\ &-3+y_{11}+y_{12}+y_{14}+y_{23} , \  -3+y_{11}+y_{12}+y_{14}+y_{33} , \  -3+y_{11}+y_{12}+y_{23}+y_{24} , \  \\ &  -5+y_{11}+y_{12}+y_{23}+y_{34} , \   -3+y_{11}+y_{12}+y_{24}+y_{33} , \  -5+y_{11}+y_{12}+y_{33}+y_{34} , \  \\ &  -4+y_{11}+y_{13}+y_{14}+y_{22} , \  -1+y_{11}+y_{13}+y_{14}+y_{32} , \    -4+y_{11}+y_{13}+y_{22}+y_{24} , \  \\ &  -6+y_{11}+y_{13}+y_{22}+y_{34} , \  -1+y_{11}+y_{13}+y_{24}+y_{32} , \  -3+y_{11}+y_{13}+y_{32}+y_{34} , \  \\ &    -7+y_{11}+y_{14}+y_{22}+y_{23} , \  -7+y_{11}+y_{14}+y_{22}+y_{33} , \  -4+y_{11}+y_{14}+y_{23}+y_{32} , \  \\ &  -4+y_{11}+y_{14}+y_{32}+y_{33} , \    -7+y_{11}+y_{22}+y_{23}+y_{24} , \  -9+y_{11}+y_{22}+y_{23}+y_{34} , \  \\ &  -7+y_{11}+y_{22}+y_{24}+y_{33} , \  -9+y_{11}+y_{22}+y_{33}+y_{34} , \    -4+y_{11}+y_{23}+y_{24}+y_{32} , \  \\ &  -6+y_{11}+y_{23}+y_{32}+y_{34} , \  -4+y_{11}+y_{24}+y_{32}+y_{33} , \  -6+y_{11}+y_{32}+y_{33}+y_{34} , \  \\ &   -1+y_{12}+y_{13}+y_{14}+y_{21} , \  y_{12}+y_{13}+y_{14}+y_{31} , \  -1+y_{12}+y_{13}+y_{21}+y_{24} , \  \\ &  -3+y_{12}+y_{13}+y_{21}+y_{34} , \    y_{12}+y_{13}+y_{24}+y_{31} , \  -2+y_{12}+y_{13}+y_{31}+y_{34} , \  \\ &  -4+y_{12}+y_{14}+y_{21}+y_{23} , \  -4+y_{12}+y_{14}+y_{21}+y_{33} , \    -3+y_{12}+y_{14}+y_{23}+y_{31} , \  \\ &  -3+y_{12}+y_{14}+y_{31}+y_{33} , \  -4+y_{12}+y_{21}+y_{23}+y_{24} , \  -6+y_{12}+y_{21}+y_{23}+y_{34} , \  \\ &    -4+y_{12}+y_{21}+y_{24}+y_{33} , \  -6+y_{12}+y_{21}+y_{33}+y_{34} , \  -3+y_{12}+y_{23}+y_{24}+y_{31} , \  \\ &  -5+y_{12}+y_{23}+y_{31}+y_{34} , \    -3+y_{12}+y_{24}+y_{31}+y_{33} , \  -5+y_{12}+y_{31}+y_{33}+y_{34} , \  \\ &  -5+y_{13}+y_{14}+y_{21}+y_{22} , \  -2+y_{13}+y_{14}+y_{21}+y_{32} , \    -4+y_{13}+y_{14}+y_{22}+y_{31} , \  \\ &  -1+y_{13}+y_{14}+y_{31}+y_{32} , \  -5+y_{13}+y_{21}+y_{22}+y_{24} , \  -7+y_{13}+y_{21}+y_{22}+y_{34} , \  \\ &    -2+y_{13}+y_{21}+y_{24}+y_{32} , \  -4+y_{13}+y_{21}+y_{32}+y_{34} , \  -4+y_{13}+y_{22}+y_{24}+y_{31} , \  \\ &  -6+y_{13}+y_{22}+y_{31}+y_{34} , \    \underline{-1+y_{13}+y_{24}+y_{31}+y_{32}} , \  \underline{-3+y_{13}+y_{31}+y_{32}+y_{34}} , \  \\ &  -8+y_{14}+y_{21}+y_{22}+y_{23} , \  -8+y_{14}+y_{21}+y_{22}+y_{33} , \    -5+y_{14}+y_{21}+y_{23}+y_{32} , \  \\ &  -5+y_{14}+y_{21}+y_{32}+y_{33} , \  -7+y_{14}+y_{22}+y_{23}+y_{31} , \  -7+y_{14}+y_{22}+y_{31}+y_{33} , \  \\ &    -4+y_{14}+y_{23}+y_{31}+y_{32} , \  -4+y_{14}+y_{31}+y_{32}+y_{33} , \  -8+y_{21}+y_{22}+y_{23}+y_{24} , \  \\ &  -10+y_{21}+y_{22}+y_{23}+y_{34} , \    -8+y_{21}+y_{22}+y_{24}+y_{33} , \  -10+y_{21}+y_{22}+y_{33}+y_{34} , \  \\ &  -5+y_{21}+y_{23}+y_{24}+y_{32} , \  -7+y_{21}+y_{23}+y_{32}+y_{34} , \    -5+y_{21}+y_{24}+y_{32}+y_{33} , \  \\ &  -7+y_{21}+y_{32}+y_{33}+y_{34} , \  -7+y_{22}+y_{23}+y_{24}+y_{31} , \  -9+y_{22}+y_{23}+y_{31}+y_{34} , \  \\ &    -7+y_{22}+y_{24}+y_{31}+y_{33} , \  -9+y_{22}+y_{31}+y_{33}+y_{34} , \  -4+y_{23}+y_{24}+y_{31}+y_{32} , \  \\ &  -6+y_{23}+y_{31}+y_{32}+y_{34} , \    \underline{-4+y_{24}+y_{31}+y_{32}+y_{33}} , \  \underline{-6+y_{31}+y_{32}+y_{33}+y_{34}} ) \enspace .
  \end{split}
  \]
\end{table}

\begin{example}
  Again we consider $V\in\RR^{3\times 4}$ as in Example~\ref{exmp:arrangement}.
  For instance, for $z=\transpose{(0,1,3)}$ we have
  \begin{equation}\label{eq:exmp_tc}
    \tc\bigl(z) \ = \ \bigl\{(3,1),(3,2),(1,3),(3,3),(2,4),(3,4)\bigr\} \enspace .
  \end{equation}
  It is instrumental to locate the point $z$ in Figure~\ref{fig:arrangement}:
  It is the unique point in the intersection of the green hyperplane (column 4) with the blue hyperplane (column 3).
  In the mixed subdivision picture in Figure~\ref{fig:mixed} the point $z$ corresponds to the maximal cell with vertices $(0,0,4)$, $(0,1,3)$, $(1,0,3)$ and $(1,1,2)$.
\end{example}

\begin{remark}\label{rem:separate}
  Consider the four max-tropical linear forms $(-v^{(k)})$ corresponding to the columns of our running example matrix $V$ for \emph{separate variables}.
  That is, we choose a new set of variables for each column.
  More precisely, we consider
  \[
  \begin{array}{ll}
    (-v^{(1)}) = \max\{y_{11},y_{21}-1,y_{31}\} & (-v^{(2)}) = \max\{y_{12},y_{22}-4,y_{32}-1\}\\
    (-v^{(3)}) = \max\{y_{13},y_{23}-3,y_{33}-3\} &  (-v^{(4)}) = \max\{y_{14},y_{24},y_{34}-2\}
  \end{array} \enspace .
  \]
  Now we can look at the max-tropical polynomial $W$ in the $12$ variables $y_{11},y_{12},\dots,y_{34}$ which arises as the tropical product of these four tropical linear forms; this is shown in Table~\ref{tab:poly}.
  The tropical covector $\tc(z)$ of a point $z\in\RR^3$ agrees with the least common multiple of those monomials of $W$ at which the maximum $W(z,z,z,z)$ is attained; here we substitute $y_{11}=y_{12}=y_{13}=y_{14}=z_1$, $y_{21}=y_{22}=y_{23}=y_{24}=z_2$ etc.\ by real numbers.
  For instance, letting $z=\transpose{(0,1,3)}$ the maximum $W(z,z,z,z)=6$ is attained at the four terms underlined in Table~\ref{tab:poly}.
  Observe that the four marked terms in Table~\ref{tab:poly} are precisely those which correspond to subsets of the tropical covector shown in (\ref{eq:exmp_tc}).
  If we substitute $y_{11}=y_{12}=y_{13}=y_{14}=x_1$, $y_{21}=y_{22}=y_{23}=y_{24}=x_2$ etc.\ by indeterminates $x=(x_1,x_2,x_3)$ the resulting expression $W(x,x,x,x)$ is precisely the tropical polynomial in \eqref{eq:arrangement}.
  This latter substitution explains the relationship between the ``fine types'' and the ``coarse types'' discussed in \cite{DochtermannJoswigSanyal:2012} or, equivalently, the relationship between the tropical covectors and the coordinates in the mixed subdivision picture.
\end{remark}

\begin{figure}[tb]\centering
  \includegraphics[height=.4\textheight]{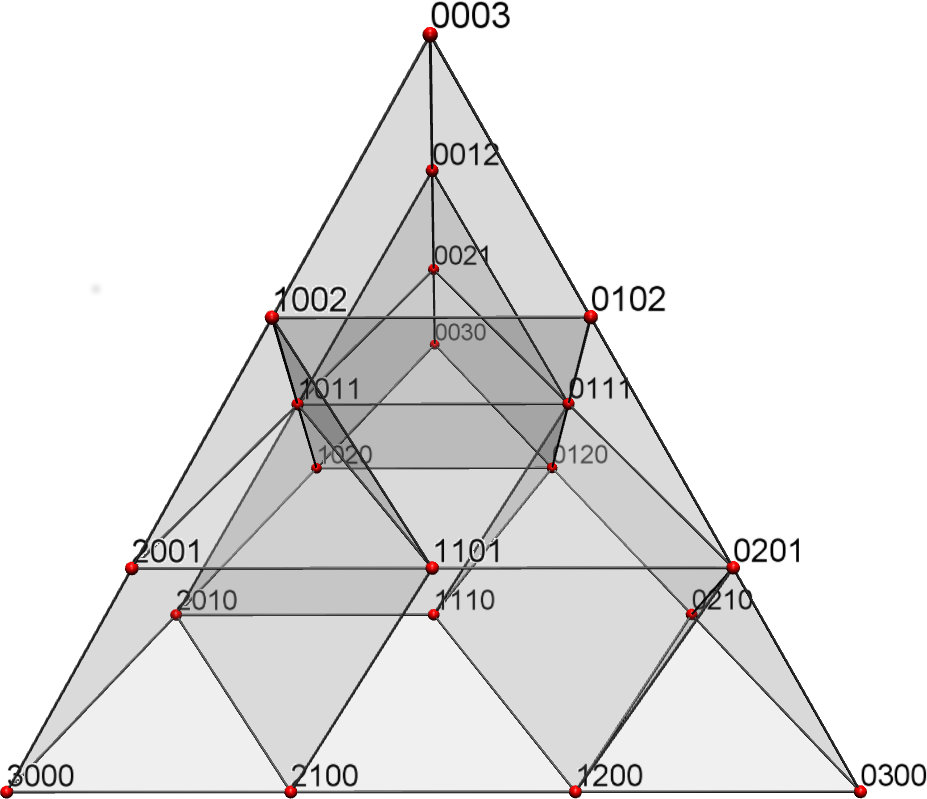}
  \caption{Mixed subdivision of $3{\cdot}\Delta_3$ dual to max-tropical hyperplane arrangement from Figure~\ref{fig:dual}}
  \label{fig:mixed-3d}
\end{figure}

\section{Ricardian Theory of International Trade}
\noindent
There is a recent interest to apply techniques from tropical geometry to questions studied in economics.
Here we focus on Shiozawa's work on international trade theory, and we summarize one part of the paper~\cite{Shiozawa:2015}.
Shiozawa suggests to describe the Ricardian theory of international trade in terms of tropical hyperplane arrangements and tropical convexity.
Since the underpinnings of that theory rely on the Cayley Trick in an essential way, it is obvious that it can be exploited.

A \emph{Ricardian economy} is described by a pair $(R,q)$ where $R=(r_{ik})$ is a $d{\times} n$-matrix of positive real numbers and $q$ is a vector of $d$ positive reals.
The $d$ rows of $R$ represent commodities, and its $n$ columns are countries.
The \emph{production coefficient} $r_{ik}$ measures how much labor is required to produce one unit of commodity $i$ in country $k$, and the number $q_k$ is the total available work force of the country $k$.
These parameters are fixed.
The purpose of this highly abstract economic model is to study the interaction of prices for the commodities and wages for the labor.
In fact, here we will focus on just a single aspect of Ricardian trade theory, which is why subsequently we will even ignore the work force vector~$q$.

A \emph{wage--price system} in this economy is a pair $(w,p)$, where $w$ is a column vector of length $n$ and $p$ is a column vector of length $d$.
Again all entries are positive.
The number $w_k$ is the wage in country $k$, and $p_i$ is the international price for commodity $i$.
Now a wage--price system $(w,p)$ is \emph{admissible} if
\begin{equation}\label{eq:ricardian:admissible} 
  r_{ik} w_k  \ \geq \ p_i \quad \text{for all } i\in[d] \text{ and } k\in[n] \enspace .
\end{equation}
These inequalities reflect the fundamental assumption that the countries compete freely among one another on the world market.
This is supposed to say that the prices are low enough to avoid excess profit.
Notice that the Ricardian economic model neglects any transport costs.

In the Equation \eqref{eq:ricardian:admissible}, for every fixed commodity $i$, we can form the minimum over all countries to obtain a total of $m$ consolidated inequalities, one for each commodity.
If we now assume that the prices are as large as possible without violating the admissibility constraints, we arrive at the equations
\begin{equation}\label{eq:ricardian:min_equation}
  \min\bigl\{r_{i1} w_1,\, r_{i2} w_2,\, \dots,\, r_{in} w_n \bigr\} \ = \ p_i \quad \text{for all } i\in[d] \enspace .
\end{equation}
Going from the inequalities \eqref{eq:ricardian:admissible} to the equations \eqref{eq:ricardian:min_equation} imposes an extra condition.
The wages are said to be \emph{sharing} for the given prices if that condition is satisfied.

It is of interest for which countries the minimum on the left of \eqref{eq:ricardian:min_equation} is attained.
The pair $(i,k)\in[d]\times[n]$ is called \emph{competitive} for the admissible wage--price system $(w,p)$ if
\[
r_{ik} w_k \ = \ p_i \ \leq \ r_{i\ell} w_\ell  \quad \text{for all } \ell\in[n] \enspace .
\]
This means that $k$ belongs to those countries that are efficient enough to produce commodity~$i$ at the international price~$p_i$.
The condition that the prices are sharing means that for each commodity $k$ there is at least one country $i$ such that $(i,k)$ is competitive for $(w,p)$.

If we now rewrite $\bar{a}_{ij}=\log r_{ij}$ as well as $\bar w_k=\log w_k$ and $\bar p_i=\log p_i$ then \eqref{eq:ricardian:min_equation} becomes a system of tropical linear equations which read
\begin{equation}\label{eq:ricardian:log_equation}
  \begin{split}
    \bar p_i \ &= \ \min\bigl\{\bar r_{i1}+\bar w_1 ,\, \bar r_{i2}+\bar w_2 ,\, \dots,\, \bar r_{in}+\bar w_n\bigr\} \\
    &= \  (\bar r_{i1}\odot\bar w_1) \oplus (\bar r_{i2}\odot\bar w_2 ) \oplus \dots \oplus (\bar r_{in}\odot\bar w_n) \quad \text{for all } i\in[d] \enspace .
  \end{split}
\end{equation}
That is, $\bar p$ is sharing if and only if $\bar p$ is contained in the tropical cone $\mintpos(\bar R)$.
Notice that in this translation we make use of the fact that the logarithm function is monotone.
Letting $\bar R=(\bar r_{ij})_{i,j}\in\RR^{d\times n}$ and similarly $\bar w=\transpose{(\bar w_1,\dots,\bar w_n)}$ as well as $\bar p=\transpose{(\bar p_1,\dots, \bar p_d)}$ we obtain \eqref{eq:ricardian:log_equation} in matrix form
\begin{equation}\label{eq:ricardian:sharing}
  \bar R \odot \bar w \ = \ \bar p \enspace .
\end{equation}

\begin{example}\label{exmp:ricardian:sharing}
  Let us consider as $\bar R$ the matrix $V$ from Example~\ref{exmp:arrangement}, i.e., $d=3$ and $n=4$.
  This way, e.g., the coefficient $v_{24}=\bar r_{24}=0$ is interpreted as the logarithmic cost to produce one unit of commodity $2$ in country $4$.
  For instance, the logarithmic wage--price system
  \begin{equation}\label{eq:ricardian:wp}
    \bar w \ = \ \transpose{(5,5,1,2)} \quad \text{and} \quad \bar p \ = \ \transpose{(1,2,4)}
  \end{equation}
  satisfies the equation \eqref{eq:ricardian:sharing}.
  Notice that $\transpose{(1,2,4)}$ and, e.g., $\transpose{(0,1,3)}$ are the same modulo $\RR\1$.
  That is, multiplying the prices and the wages by a global constant does not change the equation \eqref{eq:ricardian:sharing}.
  For this particular wage-price system the pairs
  \[
  (1,3)\,, \quad (2,4)\, , \quad (3,3) \quad \text{and} \quad (3,4)
  \]
  are competitive.
  That is, the commodity $1$ can only be produced sufficiently efficient in country $3$, while commodity $2$ is best produced in country $4$.
  The third commodity can be produced efficiently in countries $3$ and $4$.
  For these wages and prices countries $1$ and $2$ cannot compete at all.
  The logarithmic wage vector $\bar w=\transpose{(5,5,1,2)}$ is sharing for the logarithmic price vector $\bar p=\transpose{(1,2,4)}$:
  For each commodity there is at least one country that can produce sufficiently efficient to meet the prices on the world market.
  Notice that the competitive pairs form a subset of the tropical covector of the point $\transpose{(0,1,3)}$ given in~\eqref{eq:exmp_tc}.
  In fact, they correspond to the covector of the point $\bar w\in\RR^4/\RR\1$ with respect to $\transpose{V}$.
  Conceptually, this information allows to locate $\bar w$ in Figure~\ref{fig:dual}.
  Practically, however, it is a bit tedious to accomplish in a flat picture.
\end{example}

In the Ricardian economy there is a built-in symmetry between prices and wages, and this is what we want to elaborate now.
We can rewrite the admissibility condition~\eqref{eq:ricardian:admissible} as
\[
w_k \ \geq \ r_{ik}^{-1} p_i \quad \text{for all } i\in[d] \text{ and } k\in[n] \enspace .
\]
This works as we assumed that the production coefficients $r_{ik}$ are strictly positive.
We can define the matrix $R^\#=(r_{ik}^{-1})_{ki}\in\RR^{n\times d}$ and its logarithm
\[
{\bar R}^\# \ = \ (\log r_{ik}^{-1})_{ki} \ = \ (-\log r_{ik})_{ki} \ = \ -\transpose{\bar R} \enspace .
\]
Notice that, in contrast to the definition of $\bar R$, for the construction of ${\bar R}^\#$ we are taking \emph{negative} logarithms of the coefficients, and this corresponds to changing to $\max$ as the tropical addition.
Also observe that $\bar R^{\#\#}=\bar R$.
In this way the admissibility condition, in its logarithmic form, is equivalent to
\[
\bar w \ \geq \ {\bar R}^\# \maxodot \bar p \enspace .
\]
As before we impose an extra condition, namely equality in the above:
\begin{equation}\label{eq:ricardian:covering}
  \bar w \ = \ {\bar R}^\# \maxodot \bar p \enspace .
\end{equation}
In that case the prices are called \emph{covering} for the given wages.
This is dual to the sharing condition for the wages.
That is, in this case, each country can produce at least one commodity efficiently enough to be able to afford maximum wages.
In this way a wage--price system $(w,p)$ that is both sharing and covering yields the pair of equalities
\begin{align}
  \bar w \ &= \ {\bar R}^\# \maxodot (\bar R \minodot \bar w) \label{eq:w} \\
  \bar p \ &= \ \bar R \minodot ({\bar R}^\# \maxodot \bar p) \label{eq:p} \enspace .
\end{align}


Let us define the \emph{Shapley operator} of the Ricardian economy as the map $T:\RR^d\to\RR^d$ that sends a logarithmic price vector $\bar p$ to $\bar R \minodot ({\bar R}^\# \maxodot \bar p)$.
Then \eqref{eq:p} says that $\bar p$ is a fixed point of the Shapley operator $T$.
The name ``Shapley operator'' is borrowed from the theory of stochastic games; see \cite[\S2.2]{YoungZamir:2014}.
Below we will characterize the fixed points of the Shapley operator.

For two vectors $x,y\in\RR^d$ we can define
\[
\delta(x,y) \ := \ \max_{i,j} |x_i + y_j - x_j - y_i| \enspace ,
\]
and this yields a metric on the tropical projective torus $\RR^d/\RR\1$.
This is sometimes called \emph{Hilbert's projective metric}.
For an arbitrary tropical polytope $P\subset\RR^d/\RR\1$ and an arbitrary vector $v\in\RR^d$ among all vectors $w\in\RR^d$ with $w\geq v$ and $w+\RR\1\in P$ there is unique coordinatewise minimal vector $w'$; see \cite[Proposition~7]{JoswigSturmfelsYu:2007}.
The point $w'+\RR\1$ is the \emph{nearest point} to $v$ in $P$.
The value $\delta(v,w')$ minimizes the distance between $v$ and all points in~$P$.
However, in general, that minimum may be attained for other points, too; see \cite[Example~9]{JoswigSturmfelsYu:2007}.

\begin{theorem}\label{thm:shapley}
  The Shapley operator $T:\RR^d\to\RR^d$ maps a vector $\bar p$ to its nearest point in the tropical polytope $\mintconv(\bar R)$.
  In particular, the points in $\mintconv(\bar R)$ are precisely the fixed points of $T$.
\end{theorem}
\begin{proof}
  The $i$-th coefficient of the vector $T(\bar p)$ is the real number
  \[
  \begin{split}
    (\bar r_{i1},\dots,\bar r_{in}) \minodot  \transpose{\bigl(\max_{j\in[d]}\{-\bar r_{j1}+\bar p_j\},\dots,\max_{j\in[d]}\{-\bar r_{jn}+\bar p_j\}\bigr)} \\
    = \ \min_{k\in[n]} \ \max_{j\in[d]} \, \{ \bar r_{ik} - \bar r_{jk}+\bar p_j \} \enspace .
  \end{split}
  \]
  This agrees with the formula in \cite[Lemma~8]{JoswigSturmfelsYu:2007}, from which we infer that the Shapley operator sends $\bar p$ to the nearest point in the min-tropical convex hull $P:=\tconv(\bar R)$ of the columns of the matrix $\bar R$.
  For each point $x\in P$ its nearest point in $P$ is $x$ itself.
\end{proof}
As an immediate consequence the Shapley operator $T$ is idempotent, i.e., for all logarithmic price vectors $\bar p$ we have
\[
T(\bar p) \ = \ T\bigl(T(\bar p)\bigr) \enspace .
\]

In the above we first analyzed the prices and then deduced the wages.
However, this reasoning can be reversed.
The \emph{dual Shapley operator} is the map $T^\#:\RR^n\to\RR^n$ that maps a logarithmic wage vector $\bar w$ to ${\bar R}^\# \maxodot (\bar R \minodot \bar w)$.
The wages in \eqref{eq:w} can be analyzed directly by studying $T^\#$ instead of $T$, as in Theorem~\ref{thm:shapley}.

\begin{example}
  We continue the Example~\ref{exmp:ricardian:sharing}.
  Again we look at the logarithmic wage--price system $(\bar w,\bar p)$ from \eqref{eq:ricardian:wp} for the (logarithmic) production coefficients given by the $3{\times}4$-matrix $V$ from Example~\ref{exmp:arrangement}.
  We saw that the wages are sharing, but the prices are not covering since the countries $1$ and $2$ cannot successfully compete on the world market.
  Applying the dual Shapley operator $T^\#$ to $\bar w=\transpose{(5,5,1,2)}$ gives the new logarithmic wage vector
  \[
  \bar w' \ = \ \transpose{(4,3,1,2)} \enspace ,
  \]
  which now yields
  \[
  V \minodot \bar w' \ = \ \transpose{(1,2,4)} \ = \ \bar p
  \]
  from \eqref{eq:ricardian:sharing}.
  That is, lowering the logarithmic wages from $\bar w$ to $\bar w'$ gives the logarithmic wage-price system $(\bar w',\bar p)$ that is both sharing and covering.
  Notice that this only affects the wages in the countries $1$ and $2$ which could not compete previously.
  The competitive pairs are given by the covector of the point $\transpose{(1,2,4)}$, which agrees with $\transpose{(0,1,3)}$ modulo $\RR\1$, in \eqref{eq:exmp_tc}.
\end{example}

\begin{corollary}
  The wage--price systems that are both sharing and covering bijectively correspond to pairs of points in $\mintconv(\bar R)$ and $\mintconv(\transpose{\bar R})=-\maxtconv(\bar R^\#)$ that are linked via \eqref{eq:ricardian:sharing} and \eqref{eq:ricardian:covering}. 
\end{corollary}

In the language of \cite{Shiozawa:2015} the tropical polytope $\mintconv(\bar R)$ is the ``spanning core in the price-simplex'' , whereas $\maxtconv(\bar R^\#)$ is the ``spanning core in the wage-simplex''.
In that paper the mixed subdivisions of dilated simplices occur as the ``McKenzie--Minabe diagrams''; see \cite[\S9]{Shiozawa:2015}.
Points are interpreted as ``production scale vectors''.
These describe which percentage of the total work force of a country produces which commodities.
This allows, e.g, to read off the total world production.

The Cayley Trick allows for four ways to describe and to visualize the same data.
For our running example $3{\times}4$-matrix we have the initial arrangement of four tropical hyperplanes in $\RR^3/\RR\1$ in Figure~\ref{fig:arrangement} and its transpose of three tropical hyperplanes in $\RR^4/\RR\1$ in Figure~\ref{fig:dual}.
The first arrangement is dual to the mixed subdivision of $4{\cdot}\Delta_2$ in Figure~\ref{fig:mixed}, while the second is dual to the mixed subdivision of $3{\cdot}\Delta_3$ in Figure~\ref{fig:mixed-3d}.

\bibliographystyle{amsplain}
\bibliography{main}

\end{document}